\documentclass[11pt]{article}
\usepackage{amsmath,amssymb,amsfonts}
\usepackage{graphicx}
\usepackage{mathabx}
\usepackage{subfig}
\usepackage{enumerate}
\usepackage{color}
\usepackage{cite}
\usepackage{hyperref}
\usepackage{array, boldline, makecell, booktabs}
\usepackage{multirow}

\makeatletter
\newtheorem{theorem}{Theorem}
\newtheorem{lemma}[theorem]{Lemma}
\newtheorem{corollary}[theorem]{Corollary}
\newtheorem{definition}[theorem]{Definition}

\newenvironment{proof}{\noindent\textbf{Proof.}}{\hfill $\square$}

\allowdisplaybreaks \numberwithin{equation}{section}

\title{Quality-Bayesian approach to inverse acoustic source problems with partial data}
\author{Zhaoxing Li \thanks{
School of Mathematical Sciences,
University of Electronic Science and Technology of China,
Chengdu, 611731, China.
\texttt{E-mail:~lzx130682@163.com.} The research of Li is partially supported by  NNSF of China under Grant 11771068.
}
\and Yanfang Liu \thanks{
Department of Mathematical Sciences,
Michigan Technological University,
Houghton, MI 49931, U.S.A.
\texttt{E-mail:~yanfangl@mtu.edu.} The research of Liu is partially supported by an MTU Research Excellence Fund.
}
\and Jiguang Sun
\thanks{Corresponding author:
Department of Mathematical Sciences,
Michigan Technological University,
Houghton, MI 49931, U.S.A.
\texttt{E-mail:~jiguangs@mtu.edu.}
}
\and Liwei Xu \thanks{
School of Mathematical Sciences,
University of Electronic Science and Technology of China,
Chengdu, 611731, China.
\texttt{E-mail:~xul@uestc.edu.cn.} The research of Xu is partially supported by  NNSF of China under Grant 11771068.
}
}
\date{}

\begin{document}

\maketitle

\begin{quote}
\small
{\bf Abstract:}\quad
A quality-Bayesian approach, combining the direct sampling method and the Bayesian inversion, is proposed to reconstruct the locations and intensities of the unknown acoustic sources using partial data. First, we extend the direct sampling method by constructing a new indicator function to obtain the approximate locations of the sources. The behavior of the indicator is analyzed. Second, the inverse problem is formulated as a statistical inference problem using the Bayes' formula. The well-posedness of the posterior distribution is proved. The source locations obtained in the first step are coded in the priors. Then an Metropolis-Hastings MCMC algorithm is used to explore the posterior density. Both steps use the same physical model and measured data. Numerical experiments show that the proposed method is effective even with partial data.


\end{quote}

\section{Introduction }
We consider the inverse problem to reconstruct the locations and intensities of the acoustic sources from measured near-field or far-field partial data. 
The problem has importance in various applications such as biomedical imaging and the identification of pollution sources \cite{isakov2006book,eller2009}.
We refer the readers to \cite{isakov2006book,acosta2012, liu2018extended, alzaalig2017fast} 
and the references therein for different methods in literature.

Recently, a new quality-Bayesian approach, which combines the qualitative method and the Bayesian inference, 
is proposed for the inverse scattering problem \cite{li2019limited}.
The method has two steps. First, qualitative information  of the obstacle such as the size and location is obtained 
using the extended sampling method \cite{liu2018extended, liu2019extended}. 
Second, the inverse problem is formulated as a statistical inference problem using the Bayes' formula \cite{kaipio2015statistical, stuart2010}. 
Then a Markov chain Monte Carlo (MCMC) algorithm is used to explore the posterior density.
The information obtained in the first step is coded in the priors. 
The method provides satisfactory results with partial data.

In this paper, we extend the methodology in \cite{li2019limited} to the inverse acoustic source problem with partial data.
In the first step, a direct sampling method (DSM) is employed to approximate the locations of the sources.
The DSM is a non-iterative method to reconstruct the unknown scatterers \cite{potthast2006, ito2012dsm}. 
Recently, the DSM was used in \cite{Abdelaziz2015JMAA, zhang2019} 
to reconstruct the locations of multiple multipolar sources from single-frequency measurement Cauchy data.
Using multiple frequency data, we construct a new indicator function for the inverse source problems. Similar to those in \cite{Liu2017IP, alzaalig2017fast, zhang2019},
the indicator decays as the Bessel function when the sampling point moves away from the source (see, e.g., \cite{Liu2017IP}). 
The DSM has some attractive features: (\romannumeral1) it is easy to implement and computationally cheap;
(\romannumeral2) the algorithm does not need a priori information of the sources; and (\romannumeral3) 
the method can provide us with an accurate and reliable location estimation of the unknown sources; 
(\romannumeral4) the method is robust for the noise data. These features make the DSM a good candidate
to obtain some rough information of the sources.

In the second step, for more detailed properties of the sources, we resort to the Bayesian inversion.
In Bayesian statistics, the parameters are viewed as random variables. 
The number and locations of the sources obtained in the first step are coded in the priors. 
By using Bayes' formula, a posterior distribution of the unknown parameters is obtained. 
The well-posedness is proved and a Metropolis-Hastings (MH) MCMC algorithm is used to explore the posterior density. 
Consequently, statistical estimates for the unknown parameters can be obtained.
We refer the readers to \cite{kaipio2015statistical,stuart2010} for the 
Bayesian inversion and \cite{BaussardEtal2001IP, BuiGhatts2014SIAMUQ, li2019limited, liu2019inverse} for its applications to inverse problems.


The rest of paper is organized as follows. In Section 2, we present the direct and inverse source problems under investigation. 
In Section 3, we develop a DSM to obtain the approximate locations of the unknown sources. 
In Section 4, the Bayesian method is employed to reconstruct the details of the sources. 
In Section 5, numerical examples are presented to show the effectiveness of the quality-Bayesian approach using complete or partial data.
Finally, In Section~\ref{conclusions}, we draw some conclusions.

\section{Direct and Inverse Source Problems}
Let $F\in L^2(\mathbb{R}^2)$ denote the source with $\text{supp}\, F \subset V$, 
where $V$ is a bounded domain of $\mathbb{R}^2$. 
The time-harmonic acoustic wave $u\in H^1_{loc}(\mathbb{R}^2)$  radiated by $F$ satisfies
\begin{subequations}\label{Helmholtz1}
\begin{equation}\label{Helmholtz2}
\Delta u+k^2 u=F \quad \text{in}\; \mathbb{R}^2,
\end{equation}
\begin{equation}\label{Sommerfeld}
\lim\limits_{r \rightarrow \infty} \sqrt{r}
\left(\frac{\partial{u}}{\partial r}-i k u\right)=0, \quad r=|x|,
\end{equation}
\end{subequations}
where $k$ is the wave number, \eqref{Helmholtz2} is the Helmholtz equation and \eqref{Sommerfeld} is the Sommerfeld radiation condition.
Recall that the fundamental solution of the Helmholtz equation is given by
\begin{equation*}
\Phi_k(x,y)=
\dfrac{i}{4}H^{(1)}_0(k|x-y|),
\end{equation*}
where $H_0^{(1)}$ is the Hankel function of zeroth order and the first kind. It is well-known that $\Phi_k(x,y)$ 
satisfies \cite{colton2013book}
\begin{equation}\label{fundamentalsol2}
(\Delta +k^2) \Phi_k(x,y)=-\delta(x-y),
\end{equation}
where $\delta$ is the Dirac distribution.
The solution $u$ of \eqref{Helmholtz1} has the asymptotic expansion \cite{colton2013book}
\begin{equation*}
u(x,k)=\frac{e^{i\frac{\pi}{4}}}{\sqrt{8k\pi}}
\frac{e^{ikr}}{\sqrt{r}}
\left\{u^{\infty}(\hat{x},k)+\mathcal{O}\left(\frac{1}{r}\right)\right\}
\quad \text{as}\; r\rightarrow\infty,
\end{equation*}
where $\hat{x}={x}/{|x|}\in\mathbb{S}$, $\mathbb{S}:=\{|\hat{x}| = 1 : \hat{x} \in \mathbb R^2\}$. 
The function $u^{\infty}(\hat{x},k)$, $\hat{x} \in \mathbb{S}$, is the far field pattern of $u(x,k)$.
The solution $u$ to \eqref{Helmholtz1} and its far-field $u^{\infty}$ can be written as \cite{colton2013book}
\begin{equation}\label{near_field}
u({x},k)=\int_{V}\Phi_k(x,y) F(y)dy,
\end{equation}
\begin{equation}\label{far_field}
u^{\infty}(\hat{x},k)=\int_{V}\Phi^{\infty}_k(\hat{x},y) F(y)dy,
\end{equation}
where
\begin{equation}
\Phi^{\infty}_k(\hat{x},y)=\exp{(-ik\hat{x}\cdot y)},
\end{equation}
is the far-field pattern of $\Phi_k(x,y)$.

The inverse source problem (ISP) of interest is to determine $F$ from one of the following data sets:
\begin{itemize}
 \item[{\romannumeral1})]  $\left\{u(x,k): x\in \Gamma,\; k\in[k_m, k_M]\right\}$;
 \item[{\romannumeral2})]  $\left\{u^{\infty}(\hat{x},k): \hat{x}\in S \subset \mathbb{S},
 \;k\in[k_m, k_M]\right\}$;
\end{itemize}
where $\Gamma$ is a measurement curve outside $V$, and $k_m < k_M$ are two fixed wave numbers.
By complete data, we mean that $\Gamma$ is a simple closed curve with $V$ inside, or $S = \mathbb S$. 
Otherwise, the measured data is partial.

In this paper, we consider two $F$'s for \eqref{Helmholtz1}. The first one is the combination of monopole and dipole sources (see, e.g.,\cite{zhang2019})
\begin{equation}\label{F1}
F_j(x)= ( \lambda_j+\xi_j\cdot \nabla)\delta(x-z_{j}), \quad j=1, \ldots, J,
\end{equation}
where $J$ is the number of sources, $z_j$'s represent the source locations, 
$\lambda_j$'s and $\xi_j$'s are the scalar and vector intensities such that $|\lambda_j|+|\xi_j|\neq0$ and $|\lambda_j\xi_j|=0$.
The second one is the combination of 
\begin{equation}\label{F2}
F_j(x)=\sum\limits_{j=1}^{J} \lambda_j\exp(-\xi_j  |x-z_{j}|^2), \quad j=1, \ldots, J,
\end{equation}
where both $\lambda_j$'s and $\xi_j$'s are the scalar intensities.
For simplicity, we write
\begin{equation}\label{Fx}
F(x;z, {\boldsymbol \lambda}, {\boldsymbol \xi})=\sum\limits_{j=1}^J F_j(x;z_j,\lambda_j, \xi_j),\quad \text{supp}\,
F_j(x;z_j,\lambda_j, \xi_j)\subset V_j,
\end{equation}
where ${\boldsymbol \lambda}=(\lambda_1,\ldots, \lambda_J)$, ${\boldsymbol \xi}=(\xi_1, \ldots, \xi_J)$, 
$V_j \subset V$. 

\section{Direct Sampling Method}

Given the measured far-field or near-field data, we propose a direct sampling method (DSM) to reconstruct the number and locations of the sources.
The method only involves numerical integrations and thus is very fast in general.
Moreover, it is robust for noisy partial data, which makes it a good candidate to obtain some qualitative information.

Let $D$ be the sampling domain such that $V \subset D$.
Let $n$ be the unit outward normal to $\partial D$.
Denote by $\mathcal{A}:=\{k_i\}_{i=1}^N\subset [k_m,k_M]$ a finite set of discrete wave numbers.
We define two functions
\begin{equation}\label{auxiliaryeq1}
I(z_p)=\sum\limits_{k_i\in\mathcal{A} }\int_{\Gamma}
u(x,k_i)\bar{\Phi}_{k_i}(x,z_p) ds(x),\quad z_p \in D,
\end{equation}
\begin{equation}\label{auxiliaryeq2}
I^{\infty}(z_p)=\sum\limits_{k_i\in\mathcal{A} }\int_{\mathbb{S}}
u^{\infty}(\hat{x},k_i)\bar{\Phi}^{\infty}_{k_i}(\hat{x},z_p) ds(\hat{x}),\quad z_p\in D,
\end{equation}
where $\bar{\Phi}_{k_i}(x,z_p)$ and $\bar{\Phi}^{\infty}_{k_i}(\hat{x},z_p)$ are the conjugates of 
${\Phi}_{k_i}(x,z_p)$ and ${\Phi}^{\infty}_{k_i}(\hat{x},z_p)$, respectively.

\subsection{Near-field Indicator}
For near-field data, inserting \eqref{near_field} into \eqref{auxiliaryeq1} and changing the order of integration, we have that
\begin{equation}\label{secDSM_eq1}
\begin{split}
I(z_p)&=\sum\limits_{k_i\in\mathcal{A} }\int_{\Gamma}\int_V\Phi_{k_i}(x,y)F(y)dy
\bar{\Phi}_{k_i}(x,z_p) ds(x)\\
&=\sum\limits_{k_i\in\mathcal{A} }
\int_V\int_{\Gamma}\Phi_{k_i}(x,y)\bar{\Phi}_{k_i}(x,z_p) ds(x) F(y)dy.
\end{split}
\end{equation}
From \eqref{fundamentalsol2}, for $\Phi_{k_i}(x,y)$ and $\bar{\Phi}_{k_i}(x,z_p)$, it holds that
\begin{equation}\label{secDSM_eq3}
\int_{D}(  \Delta \Phi_{k_i}(x,y)+k^2 \Phi_{k_i}(x,y)  )
\bar{\Phi}_{k_i}(x,z_p)dx=-\bar{\Phi}_{k_i}(y,z_p),
\end{equation}
\begin{equation}\label{secDSM_eq4}
\int_{D}(  \Delta \bar\Phi_{k_i}(x,z_p)+k^2 \bar\Phi_{k_i}(x,z_p)  )
{\Phi}_{k_i}(x,y)dx=-{\Phi}_{k_i}(y,z_p).
\end{equation}
Using \eqref{secDSM_eq3}, \eqref{secDSM_eq4}, the Green's formula and the Sommerfeld radiation condition, we derive
\begin{equation}\label{secDSM_eq5}
\begin{split}
\quad{\Phi}_{k_i}(y,z_p)-\bar{\Phi}_{k_i}(y,z_p)
&=\int_{\Gamma}\left\{  \frac{\partial \Phi_{k_i}(x,y)}{\partial n}\bar\Phi_{k_i}(x,z_p)
-\frac{\partial \bar\Phi_{k_i}(x,z_p)}{\partial n}\Phi_{k_i}(x,y)   \right\}ds(x)\\
&\approx\int_{\Gamma}(  ik_i\Phi_{k_i}(x,y) \bar\Phi_{k_i}(x,z_p)
+ik_i \bar \Phi_{k_i}(x,z_p) \Phi_{k_i}(x,y)  )ds(x),
\end{split}
\end{equation}
which implies
\begin{equation}\label{secDSM_eq6}
\int_{\Gamma}k_i\Phi_{k_i}(x,y)\bar\Phi_{k_i}(x,z_p)ds(x)\approx Im(\Phi_{k_i}(y,z_p)).
\end{equation}

For $y\in V$ and $z_p\in D$, define the kernel function for \eqref{secDSM_eq1} 
\begin{equation}\label{kernelH}
H(y,z_p)=\int_{\Gamma}\Phi_{k_i}(x,y)\bar{\Phi}_{k_i}(x,z_p) ds(x), \quad k_i\in\mathcal{A}.
\end{equation}
Thus, 
\begin{equation}\label{Hpropto}
H(y,z_p) \,\propto \, \frac{1}{k_i}J_0(k_i|y-z_p|), \quad k_i\in \mathcal{A},
\end{equation}
where $J_0$ is the zeroth order Bessel function and $\propto$ means ``proportional to".

Due to \eqref{secDSM_eq1} and \eqref{Hpropto}, $I(z_p)$ is a superposition of the Bessel functions.
From the asymptotic property of $J_0(t)$ \cite{colton2013book}
\begin{equation}\label{J_0t}
J_0(t)=\frac{\sin t+\cos t}{\sqrt{\pi t} }\left\{1+\mathcal{O}\left(\frac{1}{t}\right)  \right\},\quad \text{as}\; t\rightarrow \infty,
\end{equation}
$I(z_p)$ decays similarly when $z_p$ moves away from $y$.

Now we define the indicator for near-field data
\begin{equation}\label{I_DSM}
{I}_{DSM}(z_p)=\frac{\Big| \sum\limits_{k_i\in\mathcal{A}}
\langle u(x,k_i),\Phi_{k_i}(x,z_p) \rangle_{L^2(\Gamma)}\Big|}
{\sum\limits_{k_i\in\mathcal{A}} \|u(x,k_i)\|_{L^2(\Gamma)}
\|\Phi_{k_i}(x,z_p)\|_{L^2(\Gamma)}}, \quad \forall z_p\in D,
\end{equation}
where the inner product $\langle \cdot,\cdot \rangle_{L^2(\Gamma)}$ is defined as
\begin{equation*}
\langle u(x,k_i),\Phi_{k_i}(x,z_p) \rangle_{L^2(\Gamma)}=
\int_{\Gamma}u(x,k_i)\bar{\Phi}_{k_i}(x,z_p) ds(x).
\end{equation*}

\subsection{Far-field Indicator}
Substituting \eqref{far_field} into \eqref{auxiliaryeq2}, and changing the order of integration, we have that
\begin{equation}\label{secDSM_eq8}
\begin{split}
I^{\infty}(z_p)&=\sum\limits_{k_i\in\mathcal{A} }
\int_{S}\int_V\Phi^{\infty}_{k_i}(\hat{x},y)F(y)dy
\bar{\Phi}^{\infty}_{k_i}(\hat{x},z_p) ds(\hat{x})\\
&=\sum\limits_{k_i\in\mathcal{A} }
\int_V\int_{S}e^{-ik_i\hat{x}\cdot y}e^{ik_i\hat{x}\cdot z_p} ds(\hat{x}) F(y)dy.
\end{split}
\end{equation}
For $y\in V$ and $z_p\in D$, define the kernel function
\begin{equation}\label{secDSM_eq9}
H^{\infty}(y,z_p)=\int_{S}e^{ik_i\hat{x}\cdot( z_p- y)}  ds(\hat{x}), \quad k_i\in \mathcal{A}.
\end{equation}
By Funk-Hecke formula \cite{colton2013book},
\begin{equation}\label{H_infty}
H^{\infty}(y,z_p)=2\pi J_0(k_i|y-z_p|), \quad k_i\in \mathcal{A}.
\end{equation}
Due to \eqref{J_0t}, $I^{\infty}(z_p)$ is large when $z_p\in V$ 
and decays when $z_p\rightarrow \infty$.
Consequently, we define the indicator for far-field data as
\begin{equation}\label{I_inftyDSM}
{I}^{\infty}_{DSM}(z_p)=\frac{ \Big| \sum\limits_{k_i\in\mathcal{A}}
\langle u^{\infty}(\hat{x},k_i),\Phi^{\infty}_{k_i}(\hat{x},z_p) \rangle_{L^2(S)}\Big| }
{\sum\limits_{k_i\in\mathcal{A}} \|u^{\infty}(\hat{x},k_i)\|_{L^2(S)}
\|\Phi^{\infty}_{k_i}(\hat{x},z_p)\|_{L^2(S)}},
\end{equation}
where
\begin{equation*}
\langle u^{\infty}(\hat{x},k_i),\Phi^{\infty}_{k_i}(\hat{x},z_p) \rangle_{L^2(S)}=
\int_{S}u^{\infty}(\hat{x},k_i)\bar{\Phi}^{\infty}_{k_i}(\hat{x},z_p) ds(\hat{x}).
\end{equation*}

The algorithm to approximate the locations of the unknown sources is as follows.

\vskip 2mm
{\bf DSM for ISP}
\begin{itemize}
  \item[1.] Collect the data $u(x,k_i)$ on $\Gamma$ (or $u^{\infty}(\hat{x},k_i)$ on $S$) for $k_i\in \mathcal{A}$.
  \item[2.] Generate sampling points for $D$.
  \item[3.] Calculate ${I}_{DSM}(z_p)$ (or ${I}^{\infty}_{DSM}(z_p)$) for all sampling points $z_p\in D$.
  \item[4.] Find the local maximizers $z^{DSM}$'s of ${I}_{DSM}(z_p)$ (or ${I}^{\infty}_{DSM}(z_p)$).
\end{itemize}
The DSM provides the number and locations of the sources. Such qualitative information is integrated into
the priors for the following Bayesian inversion.

\section{Bayesian Inversion}
In this section, we present the Bayesian inversion for the inverse acoustic source problem using the near-field data. 
The far-field case is similar. Let 
${\boldsymbol \phi}=({\boldsymbol \lambda}, {\boldsymbol \xi}, z_1, \ldots, z_J)^T$.
Define the forward operator $\mathcal{K}: \mathbb{R}^{N}\rightarrow L^2(\Gamma)$ as follows
\begin{equation*}
\mathcal{K}({\boldsymbol \phi}):=\int_{V}\Phi_k(y,x) F(x; {\boldsymbol \phi})dx.
\end{equation*}
For $F$ given by \eqref{F1}, $N=5J$. For $F$ given by \eqref{F2}, $N=4J$. 

The statistical modal for \eqref{Helmholtz1} can be written as
\begin{equation*}
Y=\mathcal{K}({\boldsymbol \phi})+ Z,
\end{equation*}
where $Y$ is the noisy measurement of $u(x,k)$ and $Z$ is the Gaussian noise, i.e., $Z \sim \mathcal{N}(0,\gamma^2 I)$. 

Denote by $\mu^Y(d{\boldsymbol \phi})=\mathbb{P}(d {\boldsymbol \phi}  |Y)$ the posterior measure and 
by $\mu_0(d{\boldsymbol \phi})= \mathbb{P}(d {\boldsymbol \phi})$ the prior measure for ${\boldsymbol \phi}$. 
The statistical inverse problem is to find the posterior measure $\mu^Y(d{\boldsymbol \phi})$.
Assume that $\mu^Y$ is absolutely continuous with respect to $\mu_0$, i.e., $\mu^Y\ll \mu_0$. By Bayes' formula,
\begin{equation}\label{Bayesrule}
\frac{d \mu^Y}{d \mu_0}({\boldsymbol \phi})=
\frac{1}{L(Y)}\exp\left(-G({\boldsymbol \phi};Y)\right),
\end{equation}
where
$\exp(-G({\boldsymbol \phi};Y))$
is the likelihood such that
\begin{equation}\label{likelipptential}
G({\boldsymbol \phi};Y):=\frac{1}{2\gamma^2}\|Y-\mathcal{K}({\boldsymbol \phi})\|^2_{L^2(\Gamma)},
\end{equation}
and
\begin{equation}\label{normalizingconst}
L(Y):=\int_{\mathbb{R}^N} \exp\left(-G({\boldsymbol \phi};Y)) d\mu_0({\boldsymbol \phi}\right),
\end{equation}
is the normalizing constant of $\mu^Y$.


In the rest of this section, we show the well-posedness of the posterior measure following~\cite{stuart2010}. 
We first analyze the forward operator $\mathcal{K}$.
\begin{lemma} (Page 441 of \cite{watson1952book})
Let $z\in \mathbb{C}$ and $Re(z)>0$. Then the following Nicholson's formula holds
\begin{equation}\label{Nicholson}
J_{\nu}^2(z)+Y_{\nu}^2(z)=\frac{8}{\pi^2}\int_0^{\infty}
K_0(2 z \sinh t)\cosh(2\nu t) dt,
\end{equation}
where
\begin{equation*}
K_0(z)=\int_0^{\infty} e^{-z \cosh t} dt,
\end{equation*}
$J_{\nu}(z)$ and $Y_{\nu}(z)$ are the $\nu$-th order Bessel function and Neumann function, respectively.
\end{lemma}

Using Nicholson's formula \eqref{Nicholson}, the following property holds for $\mathcal{K}$.
\begin{lemma}\label{thmdirect}
There exist a constant $C$, such that
\begin{equation}\label{direct_theorem1}
\|\mathcal{K}({\boldsymbol \phi})\|_{L^2(\Gamma)}\leqslant C |{\boldsymbol \phi}|_\infty,
\end{equation}
where $C$ depends on $\Gamma$ and $V$.
\end{lemma}
\begin{proof}
According to \eqref{near_field}, 
\begin{equation}\label{thmdirect_1}
u(x,k)=\int_{V} \frac{i}{4} H_0^{(1)} (k|x-y|) F(y) dy.
\end{equation}
Let
\begin{equation*}
\tau^*=\min\{|x-y|: x\in \Gamma, y\in V  \}
\quad \text{and} \quad
\tau=|x-y|.
\end{equation*}
By \eqref{Nicholson}, 
\begin{equation}\label{thmdirect_2}
\begin{split}
|H_{\nu}^{(1)}(k\tau)|^2&=J_{\nu}^2(k\tau)+Y_{\nu}^2(k\tau)\\
&= \frac{8}{\pi^2}\int_0^{\infty}
K_0(2k\tau \sinh t)\cosh(2\nu t)dt\\
&\leqslant \frac{8}{\pi^2}\int_0^{\infty}
K_0(2k\tau^* \sinh t)\cosh(2\nu t)dt\\
&=|H_{\nu}^{(1)}(k\tau^*)|^2.
\end{split}
\end{equation}
Combining \eqref{thmdirect_1} and \eqref{thmdirect_2}, we obtain that
\begin{equation}\label{thmdirect_3}
\begin{split}
|u(x,k)|&=\left|\int_{V} \frac{i}{4}
 H_0^{(1)} (k|x-y|) F(y) dy\right|\\
&\leqslant\frac{1}{4}\int_{V}
| H_0^{(1)} (k\tau^*)| |F(y)| dy \\
&\leqslant\frac{|V|}{4}
| H_0^{(1)} (k\tau^*)| \|F(y)\|_{\infty}.
\end{split}
\end{equation}
Since $k\geqslant k_m$, it holds that
\begin{equation}\label{thmdirect_4}
\|u(x,k)\|_{L^2(\Gamma)}\leqslant\frac{|V|}{4}
| H_0^{(1)} (k_m\tau^*)| \|F(y)\|_{\infty} \leqslant C\frac{|V|}{4}
| H_0^{(1)} (k_m\tau^*)| |{\boldsymbol \phi}|_{\infty}.
\end{equation}
This completes the proof.
\end{proof}

\begin{corollary}
There exists a constant $C$, such that
\begin{equation}\label{direct_theorem1}
\|\mathcal{K}({\boldsymbol \phi}_1)-\mathcal{K}({\boldsymbol \phi}_2)\|_{L^2(\Gamma)}\leqslant
C |{\boldsymbol \phi}_1-{\boldsymbol \phi}_2|_\infty,
\end{equation}
where $C$ depends on $\Gamma$ and $V$.
\end{corollary}

Let $\mu_1$ and $\mu_2$ denote two probability measures. Assume that $\mu_1$ and $\mu_2$ are both absolutely continuous with respect to a third measure $\mu$.

\begin{definition}
The Hellinger and total variation metrics between $\mu_1$ and $\mu_2$ are, respectively, defined as
\begin{equation}\label{hellinger}
d_H(\mu_1,\mu_2)=\left(  \frac{1}{2} \int
\left(  \sqrt{d\mu_1/d\mu} -\sqrt{d\mu_2/d\mu}  \right)^2 d\mu \right)^{1/2},
\end{equation}
and
\begin{equation}\label{tv}
d_{TV}(\mu_1,\mu_2)= \frac{1}{2} \int
\left|  {d\mu_1/d\mu} -{d\mu_2/d\mu}  \right| d\mu.
\end{equation}
\end{definition}

Both Hellinger and total variation metrics are independent of the choice of the measure $\mu$. 
If $\mu_1$ and $\mu_2$ are both absolutely continuous with respect to $\mu$, then
Hellinger and total variation metrics are equivalent (see, e.g., Lemma 6.36 of \cite{stuart2010}).

\begin{theorem}
Assume that $\mu_0$ is a Borel probability measure on $\mathbb{R}^N$. Let $\mu^{Y}$ and  
$\mu^{Y'}$ be the measures defined by \eqref{Bayesrule} for $Y$ and $Y' \in \mathbb{C}$, respectively. Suppose $\mu^{Y}$ and $\mu^{Y'}$ 
are both absolutely continuous with respect to $\mu_{0}$. Then the Bayesian inverse problem \eqref{Bayesrule} is well-posed in both 
Hellinger and total variational metrics, i.e.,  there exist a constant $M=M(r)>0$ with $\max\{|Y|,|Y'|\}< r$, such that
\begin{equation}\label{wellposed}
d_H(\mu^{Y},\mu^{Y'})\leq M\|Y-Y'\|_\infty \quad \text{and} \quad
d_{TV}(\mu^{Y},\mu^{Y'})\leq M\|Y-Y'\|_{\infty}.
\end{equation}
\end{theorem}

\begin{proof}
Define the normalizing constant
\begin{equation}\label{LY}
L(Y)=\int_{\mathbb{R}^N} \exp\left(-\frac{1}{2\gamma^2}\|Y-\mathcal{K}({\boldsymbol \phi})\|_{L^2({\Gamma})}^2\right) d\mu_0({\boldsymbol \phi}).
\end{equation}
Clearly,
$$0\leq L(Y)\leq 1.$$
Now we show that $L(Y)$ is strictly positive.
From Lemma \ref{thmdirect} and (\ref{LY}), it follows that
\begin{equation}\label{thmwellposed2}
\begin{split}
L(Y)&\geq \int_{B(R)} \exp\left(-\frac{1}{\gamma^2}(\|Y\|_{L^2({\Gamma})}^2+\|\mathcal{K}({\boldsymbol \phi})\|_{L^2({\Gamma})}^2\right) d\mu_0({\boldsymbol \phi})\\
&\geq\int_{B(R)} \exp(-M) d\mu_0({\boldsymbol \phi}) \\&=
\exp(-M) \mu_0(B(R) ),
\end{split}
\end{equation}
where $B(R)$ denotes a ball with a large enough radius $R$.

We conclude that $\mu_0(B(R) )>0$. To see this, consider the disjoint sets $A_n:=\{u|\, n-1\leq \|u\|< n\}$, $\forall n\in \mathbb{N} $.
Then $\{A_n\}$ are measurable and
$\sum_{n=1}^{\infty}\mu_0(A_n)=
\mu_0(\bigcup_{n=1}^{\infty} A_n)=1$. Consequently, there exist at least one of
$A_n \in \{A_n\}_{n=1}^{\infty}$
satisfies $\mu_0(A_n)\neq 0$. Combining with \eqref{thmwellposed2}, it holds that $L(Y)>0$.

Using the mean value theorem and Lemma~\ref{thmdirect}, we have
\begin{equation}\label{thmwellposed3}
\begin{split}
|L(Y)-L(Y')|&\leq \int_{\mathbb{R}^N}
\exp(-G({\boldsymbol \phi};Y))\left|G({\boldsymbol \phi};Y)-G({\boldsymbol \phi};Y')  \right|d\mu_0({\boldsymbol \phi})\\
&\leq \int_{\mathbb{R}^N} \frac{1}{2\gamma^2}
\left| \|Y-\mathcal{K}({\boldsymbol \phi})\|_{L^2(\Gamma)}^2-
\|Y'-\mathcal{K}({\boldsymbol \phi})\|_{L^2(\Gamma)}^2\right| d\mu_0({\boldsymbol \phi})\\
&\leq M\|Y-Y'\|_{\infty}.
\end{split}
\end{equation}
By the definition of $d_H$,
\begin{equation}\label{thmwellposed4}
\begin{split}
&\quad d_H^2(\mu^Y,\mu^{Y'})\\
&=\frac{1}{2} \int_{\mathbb{R}^N}\left\{
\left(\frac{\exp(-G({\boldsymbol \phi};Y))}{L(Y)}\right)^{1/2}-
\left(\frac{\exp(-G({\boldsymbol \phi};Y'))}{L(Y')}\right)^{1/2}
\right\}^2 d\mu_0({\boldsymbol \phi})\\
&=\frac{1}{2} \int_{\mathbb{R}^N}\left\{
\left(\frac{\exp(-G({\boldsymbol \phi};Y))}{L(Y)}\right)^{1/2}-
\left(\frac{\exp(-G({\boldsymbol \phi};Y'))}{L(Y)}\right)^{1/2}
\right. \\
&\qquad\quad\quad+ \left.
\left(\frac{\exp(-G({\boldsymbol \phi};Y'))}{L(Y)}\right)^{1/2}-
\left(\frac{\exp(-G({\boldsymbol \phi};Y'))}{L(Y')}\right)^{1/2}
\right\}^2 d\mu_0({\boldsymbol \phi})\\
&\leq L(Y)^{-1}\int_{\mathbb{R}^N}\left\{
{\exp\left(-\frac{1}{2}G({\boldsymbol \phi};Y)\right)}-
{\exp\left(-\frac{1}{2}G({\boldsymbol \phi};Y')\right)}
\right\}^2 d\mu_0({\boldsymbol \phi})\\
&\qquad\quad\quad +
\left|L(Y)^{-1/2}-L(Y')^{-1/2}\right|^2
\int_{\mathbb{R}^N}{\exp\left(-G({\boldsymbol \phi};Y')\right)}
d\mu_0({\boldsymbol \phi}).
\end{split}
\end{equation}
Again, using the mean value theorem and Lemma \ref{thmdirect}, we have that
\begin{equation}\label{thmwellposed5}
\begin{split}
&\quad \int_{\mathbb{R}^N}\left\{
{\exp\left(-\frac{1}{2}G({\boldsymbol \phi};Y)\right)}-
{\exp\left(-\frac{1}{2}G({\boldsymbol \phi};Y')\right)}
\right\}^2 d\mu_0({\boldsymbol \phi})\\
&\leq \int_{\mathbb{R}^N}\exp(-G({\boldsymbol \phi};Y))
\left|\frac{1}{2}G({\boldsymbol \phi};Y)-\frac{1}{2}G({\boldsymbol \phi};Y')  \right|^2d\mu_0({\boldsymbol \phi})\\
&\leq \int_{\mathbb{R}^N} \frac{1}{16\gamma^4}
\left| \|Y-\mathcal{K}({\boldsymbol \phi})\|^2-
\|Y'-\mathcal{K}({\boldsymbol \phi})\|^2\right|^2 d\mu_0({\boldsymbol \phi})\\
&\leq M\|Y-Y'\|_{\infty}^2.
\end{split}
\end{equation}
According to the boundedness of $L(Y)$ and $L(Y')$, it holds that
\begin{equation}\label{thmwellposed6}
\begin{split}
&\quad \left|L(Y)^{-1/2}-L(Y')^{-1/2}\right|^2\\
&\leq M \max\left( L(Y)^{-3},L(Y')^{-3} \right)
\left| L(Y)-L(Y') \right|^2\\
&\leq M \|Y-Y'\|_{\infty}^2.
\end{split}
\end{equation}
Combining (\ref{thmwellposed2})-(\ref{thmwellposed6}) we obtain
\begin{equation*}
d_H(\mu^{Y},\mu^{Y'})\leq M\|Y-Y'\|_\infty.
\end{equation*}
Due to Lemma 6.36 of \cite{stuart2010}, it also holds that
\begin{equation*}
d_{TV}(\mu^{Y},\mu^{Y'})\leq M\|Y-Y'\|_{\infty}.
\end{equation*}
The proof is complete.
\end{proof}

\vspace{1ex}
The solution to the Bayesian inverse problem is the posterior distribution $\mu^Y$. 
To explore  $\mu^Y$, we resort to the Markov Chain Monte Carlo (MCMC) method. 
The basic idea is to construct an ergodic Markov chain with $\mu^Y$ as the stationary distribution. 
Based on the samples generated by MCMC, various statistical estimates
such as maximum a posteriori (MAP) and conditional mean (CM) can be obtained.
In this paper, we employ the preconditioned Crank-Nicolson (pCN) Metropolis-Hastings (MH) algorithm for MCMC \cite{Cotter2013}.

\vspace{1.5ex}
{\bf pCN-MH for ISP}
\begin{itemize}
  \item[1.] Set $n \leftarrow 0$ and choose an initial value ${\boldsymbol \phi}_0$.
  \item[2.] Propose a move according to
  \begin{eqnarray*}
  \tilde{\lambda}_{j, n}&=&\left(1-{\beta}^2\right)^{1/2}{\lambda}_{j, n}+\beta W_n,\quad W_n\sim \mathcal{N}(0,1),\\
   \tilde{\xi}_{j, n}&=&\left(1-{\beta}^2\right)^{1/2}{\xi}_{j, n}+\beta W_n,\quad W_n\sim \mathcal{N}(0,1)  (\text{or } \mathcal{N}({\boldsymbol 0},I_{2\times 2})) \\
     \tilde{z}_{j, n}&=& z^{DSM}+ \sqrt{\sigma} W_n,\quad W_n\sim \mathcal{N}(z_j^{DSM},I_{2\times 2}).
  \end{eqnarray*}
  \item[3.] Compute
  \begin{equation*}
  \alpha({\boldsymbol \phi}_n,\tilde{\boldsymbol \phi}_n)=\min \left\{1, \exp\left(G\big(F(x;{\boldsymbol \phi}_n);Y\big)-
   G(F(x;\tilde{\boldsymbol \phi}_n);Y\big) \right) \right\}.
  \end{equation*}
  \item[4.] Draw $\tilde{\alpha}\sim \mathcal{U}(0,1)$. If  $\alpha({\boldsymbol \phi}_n,\tilde{{\boldsymbol \phi}}_n)\geq \tilde{\alpha}$,
      set ${\boldsymbol \phi}_{n+1}=\tilde{\boldsymbol \phi}_n$.  Else, ${\boldsymbol \phi}_{n+1}={{\boldsymbol \phi}_n}$.
  \item[5.] When $n=\text{MaxIt}$, the maximum sample size, stop. \\
  	Otherwise, set $n \leftarrow n+1$ and go to step 2.
\end{itemize}

\section{Numerical Examples}
We present four numerical examples to show the performance of the proposed quality-Bayesian approach. Let $V=[-4,4]\times[-4,4]$. 
For point sources \eqref{F1}, the solution to the forward problem \eqref{Helmholtz1} is computed directly using the formula in Section 4.1 of \cite{zhang2019}.
For extended sources \eqref{F2},  the solution to \eqref{Helmholtz1} is approximated using \eqref{near_field} or \eqref{far_field} as follows.
Generate a triangular mesh ${\mathcal T}$ for $V$ with mesh size $h$.
For example, for $ \hat{x}:=(\cos\theta,\sin\theta)$, $\theta \in [0,2\pi)$, and a fixed wave number $k$,
 the far-field data is approximated by
 \begin{equation}\label{uinftyh}
 u^{\infty}(\hat{x}; k) \approx \sum_{T \in {\mathcal T}} e^{-ik \hat{x} \cdot {y_T}} F(y_T) |T|,
 \end{equation}
 where $T \in {\mathcal T}$ is a triangle, $y_T$ is the center of $T$, and $|T|$ denotes the area of $T$.
The synthetic data is computed using a mesh with $h \approx 0.06$. 
Then 5\% random noise is added
\begin{equation*}
u^{m}(x,k):=u(x,k)+0.05 (\tilde{Z}_1+i\tilde{Z}_2)  \max |u|,
\end{equation*}
where $\tilde{Z}_1,\tilde{Z}_2\sim\mathcal{N}(0,1)$.
Note that, in the sampling stage, a different relatively coarse mesh is used to compute the solution of the forward problem \eqref{Helmholtz1}.

Examples 1 and 3 use the near field data, which are measured on a circle with radius $R=6.5$ and centered at zero. 
Examples 2 and 4 use the far-field data. 
Define
\[
S_1:=[0,2\pi), \quad S_2:=[0,\pi), \quad S_3:=[\pi/8, 5\pi/8).
\]
The near-field data are
\[
\{u(x; k_j)| x=(R\cos \theta, R\sin \theta), \theta \in S_i, i=1,2,3,  j=1, \ldots, N_k.\}
\]
The far-field data are 
\[
\{ u^\infty(\hat{x}; k_j)| \hat{x}=(\cos \theta, \sin \theta), \theta \in S_i, i=1,2,3,  j=1, \ldots, N_k.\}
\]
The aperture $S_1$ represents complete data and $S_2, S_3$ represent partial data.
Equally spaced measurement angles are used: $80$ for $S_1$, $40$ for $S_2$ and $20$ for $S_3$.
The $N_k$ wave numbers in $[k_m , k_M]$ are given by
\begin{equation}\label{ki}
k_j=k_m+(j-1)(k_M-k_m)/(N_k-1),\quad j=1,\cdots,N_k.
\end{equation}

The sampling domain for DSM is chosen to be the same as $V$ and is divided into $201\times 201$ uniformly distributed sampling points. 

\subsection{Example 1} Let $F(x)=\sum_{j=1}^{3}\lambda_{j}\delta(x-z_{j})$ with
\[
\{\lambda_{j}\}_{j=1}^3=\{6,5,7\},\quad \{z_{j}\}_{j=1}^3=\{(2,2),(-2,2),(0,-2)\}.
\]
Let $k_m=5$, $k_M=10$ and $N_k=10$.
In Fig.~\ref{Fig1}, we plot the indicator functions using data associated with $S_1, S_2, S_3$ by DSM. 
\begin{figure}[h!]
\begin{center}
\begin{tabular}{lll}
\resizebox{0.33\textwidth}{!}{\includegraphics{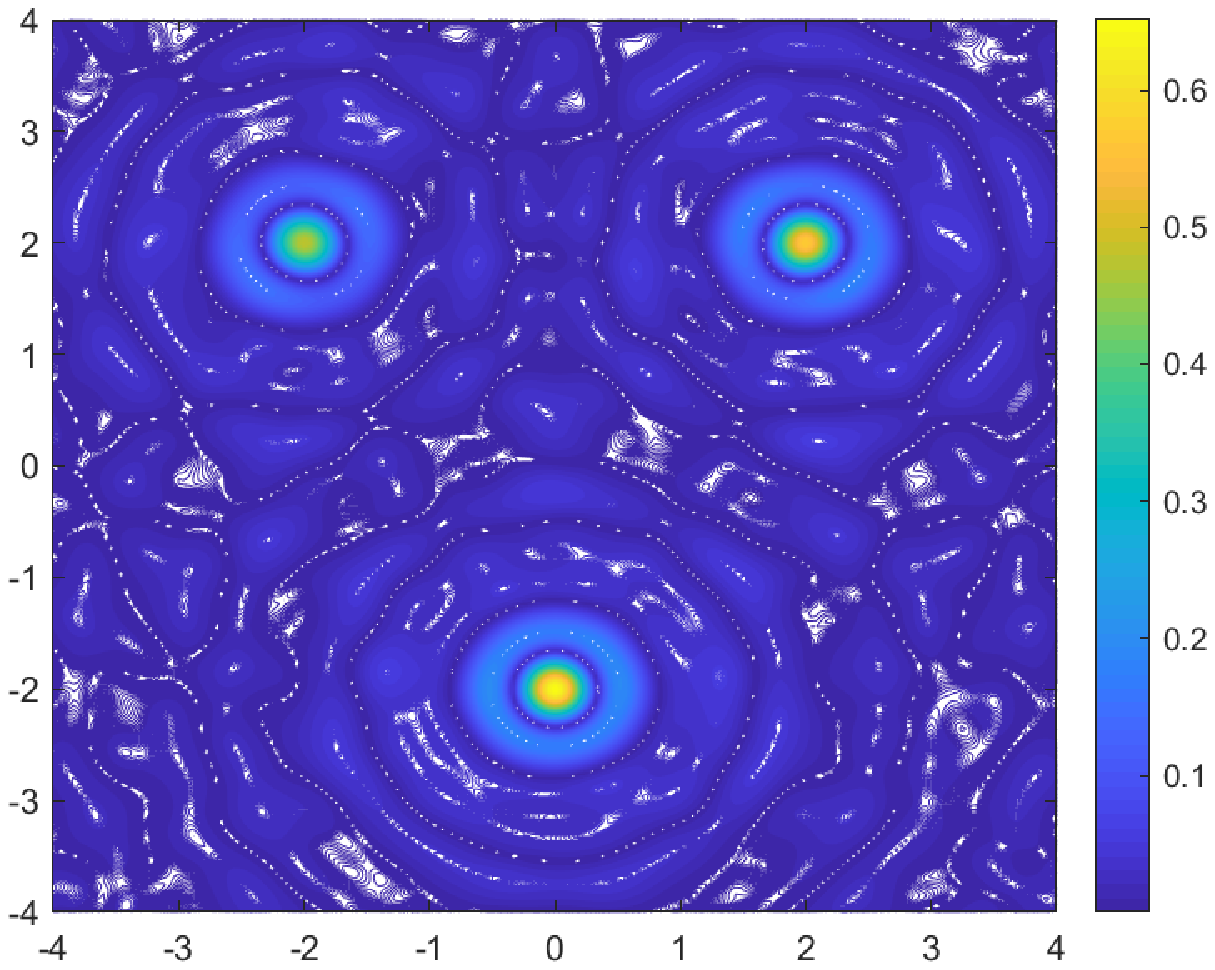}}&
\resizebox{0.33\textwidth}{!}{\includegraphics{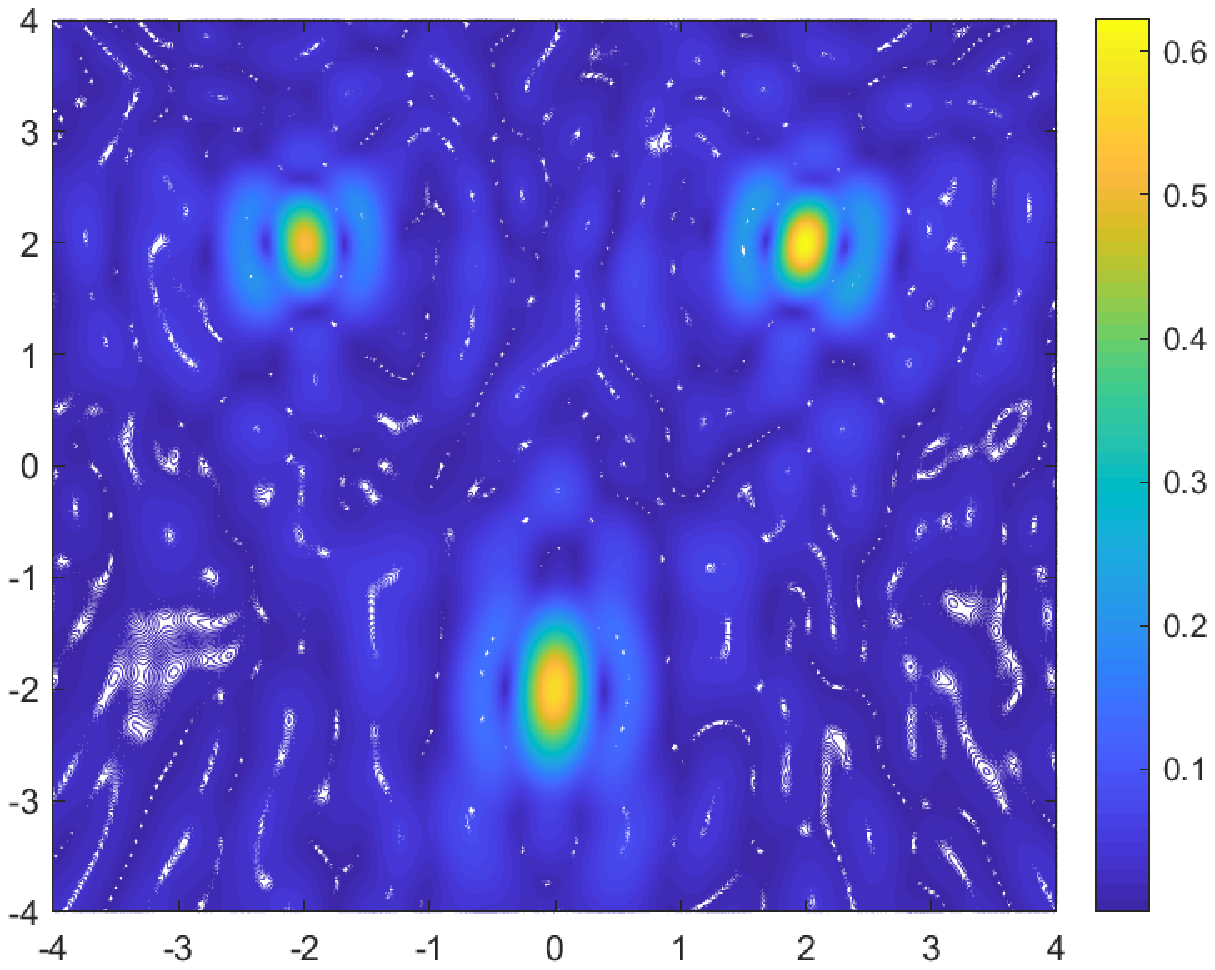}}&
\resizebox{0.33\textwidth}{!}{\includegraphics{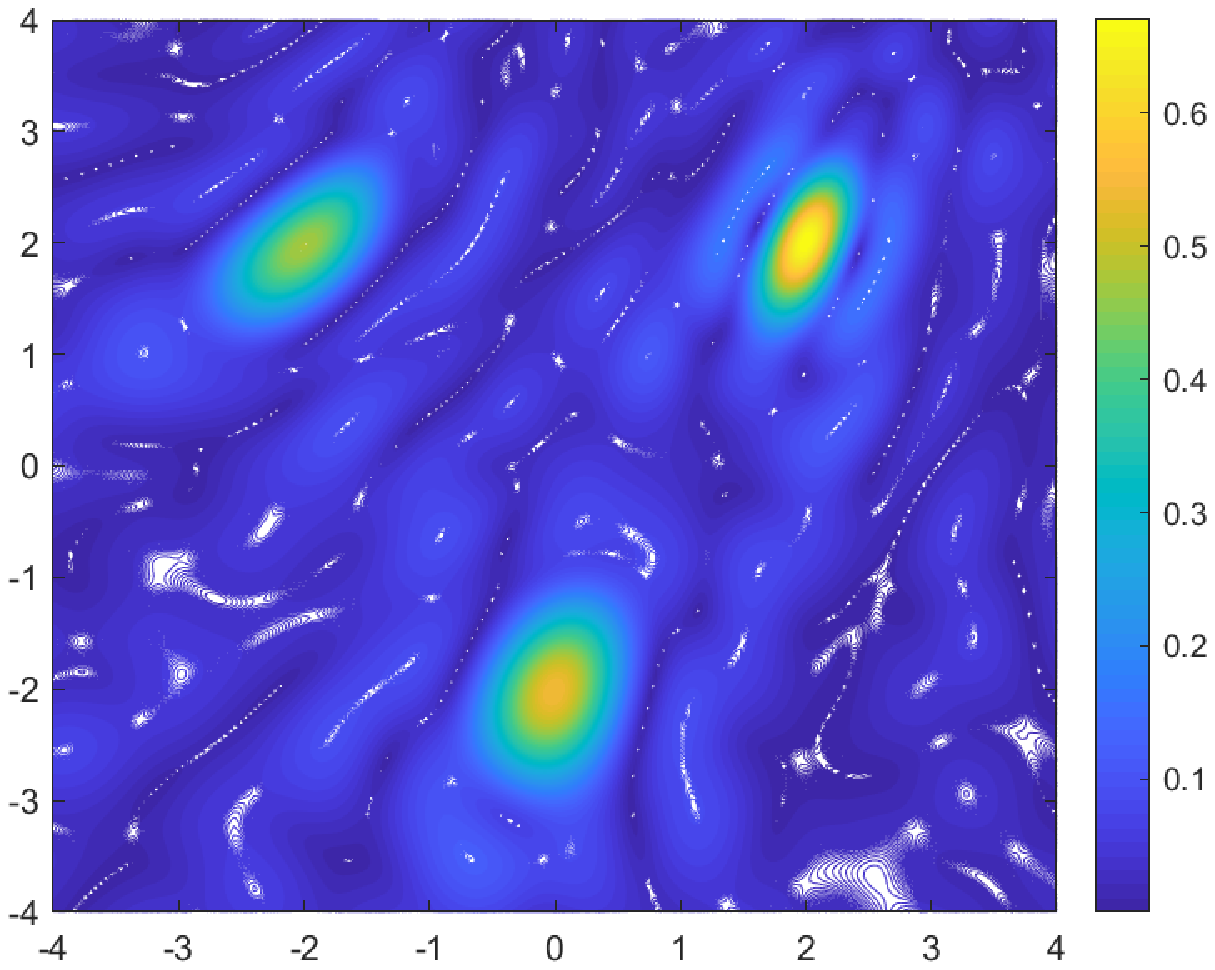}}
\end{tabular}
\end{center}
\caption{Plots of the indicators of the DSM. Left: $S_1$. 
Middle: $S_2$. Right: $S_3$.}
\label{Fig1}
\end{figure}
The reconstructed locations $\{ z_j^{DSM} \}_{j=1}^3$ are
\begin{eqnarray*}
&& S_1: \{(2.00,2.00),(-2.00,2.00),(0.00,-2.00)\}, \\
&& S_2: \{(2.00,2.00),(-2.00,2.00),(0.00,-2.00)\}, \\
&& S_3: \{(2.00,2.00),(-2.00,2.00),(0.00,-2.00)\}.
\end{eqnarray*}
For monopole sources, the locations are accurate even for partial data, although the plots of the indicators in Fig.~\ref{Fig1} are quite different.

In the Bayesian inversion stage, we take $\beta=0.03$ and $\sigma=0.0004$. 
The maximum number of samples is set to be 10000. 
The first 3000 samples are discarded and the condition mean of posterior density for each parameter is computed. 
The results are shown in Table~\ref{e1t1}, where $\cdot^\star$ represents the condition mean.
The reconstructed parameters are close to the exact ones.
\begin{table}[ht]
\small
    \centering
{   \begin{tabular}{ l  c  c  c  c  }
 \Xhline{1pt}
  \multirow{2}{*}{$j$ } & \multicolumn{2}{c}{Exact Parameters} & \multicolumn{2}{c}{Reconstructed parameters for $S_1$}  \\  
  \cmidrule(lr){2-3}  \cmidrule(lr){4-5}
 & $\lambda_{j}$ & $z_{j}$ & $\lambda_{j}^{\star}$ & $z_{j}^{\star}$ \\
 \Xhline{0.8pt}
1 & 6 & (2, 2) & 5.8528 & (1.9980, 1.9967 ) \\
  2 & 5 & (-2, 2) & 5.0681  & (-2.0026, 2.0040) \\
  3 & 7 & (0, -2) & 6.9588 & (-0.0003, -1.9982) \\
 \Xhline{1pt}
  \multirow{2}{*}{j } & \multicolumn{2}{c}{Reconstructed parameters for $S_2$} & \multicolumn{2}{c}{Reconstructed parameters for $S_3$}  \\  
  \cmidrule(lr){2-3}  \cmidrule(lr){4-5}
 & $\lambda_{j}^{\star}$ &  $z_{j}^{\star}$ & $\lambda_{j}^{\star}$ &  $z_{j}^{\star}$ \\
 \Xhline{0.8pt}
  1 & 5.9798 & (2.0000, 1.9954) & 6.2389  & (2.0103, 1.9990) \\
  2 & 5.0730 & (-2.0017, 2.0087) & 4.8071 & (-1.9949, 1.9938) \\
  3 &  6.9029 & (-0.0100, -1.9934) &  6.9066 & (0.0259, -2.0054) \\
 \Xhline{1pt}
\end{tabular} }
\caption{Exact and reconstructed parameters for Example 1.} 
\label{e1t1}
\end{table}

\subsection{Example 2} Let $F(x)=\sum_{j=1}^{3}(\lambda_{j}+\xi_{j}\cdot\nabla)\delta(x-z_{j})$ with
\begin{eqnarray*}
&&\{\lambda_{j}\}_{j=1}^3=\{0,9,0\}, \\ 
&&\{\xi_{j}\}_{j=1}^3=\{(\sqrt{2},-\sqrt{2}),(0,0),(2,0)\}, \\
&&\{z_{j}\}_{j=1}^3=\{(2,0),(-2,2),(-2,-2)\}.
\end{eqnarray*}
Again, let $k_m=5$, $k_M=10$ and $N_k=10$.
In Fig.~\ref{Fig2}, we plot the indicator functions for $S_1, S_2, S_3$. 
\begin{figure}[h!]
\begin{center}
\begin{tabular}{lll}
\resizebox{0.33\textwidth}{!}{\includegraphics{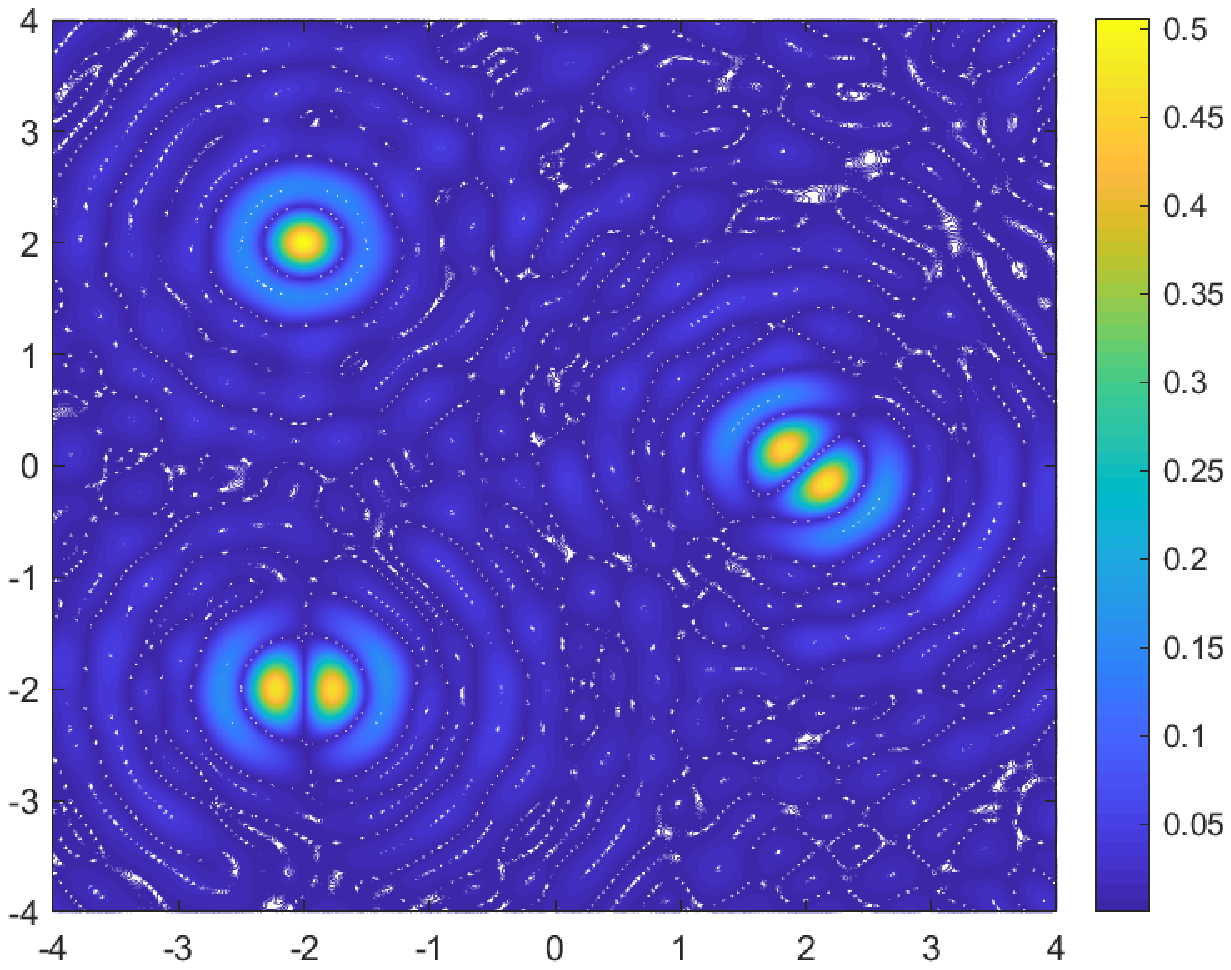}}&
\resizebox{0.33\textwidth}{!}{\includegraphics{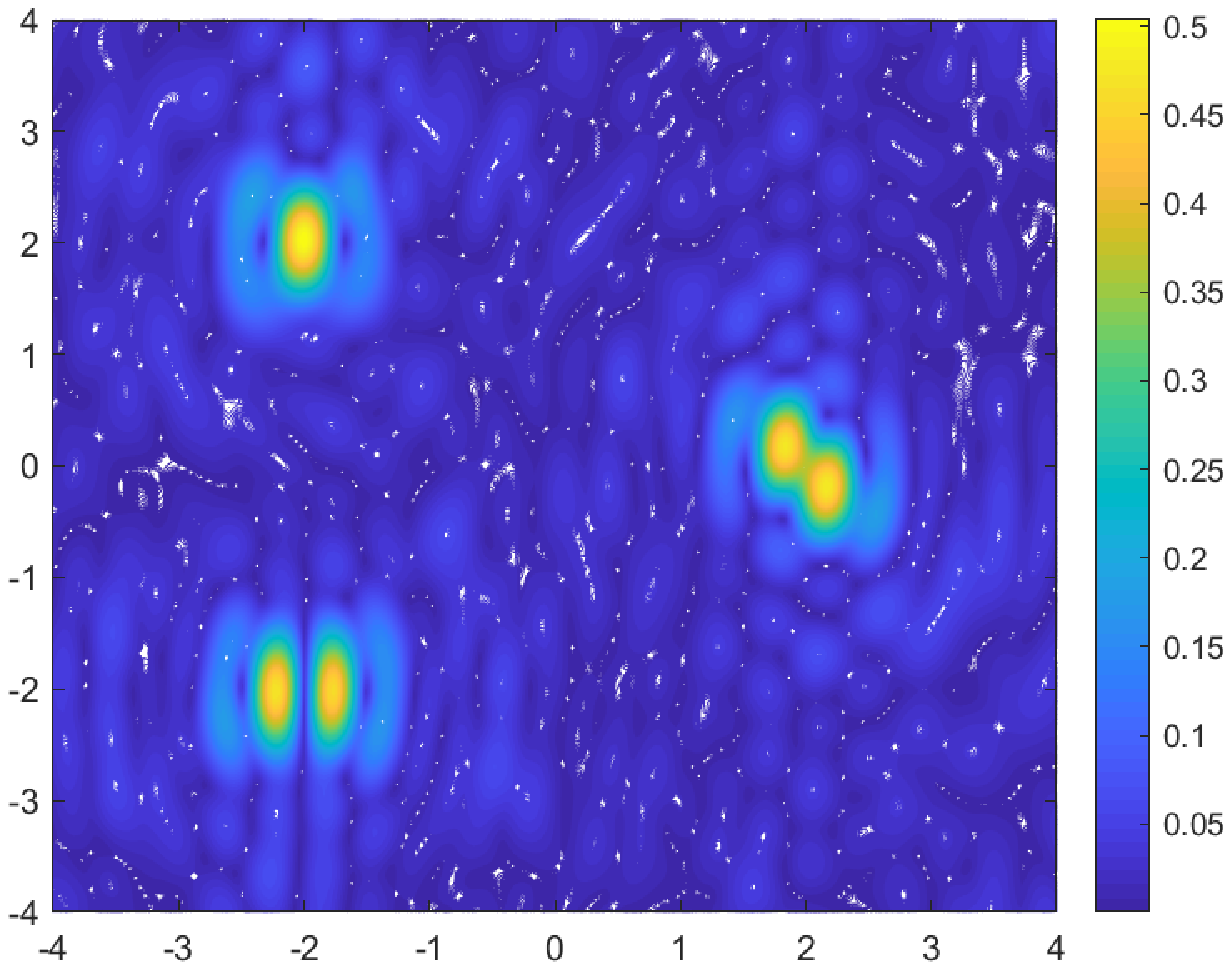}}&
\resizebox{0.33\textwidth}{!}{\includegraphics{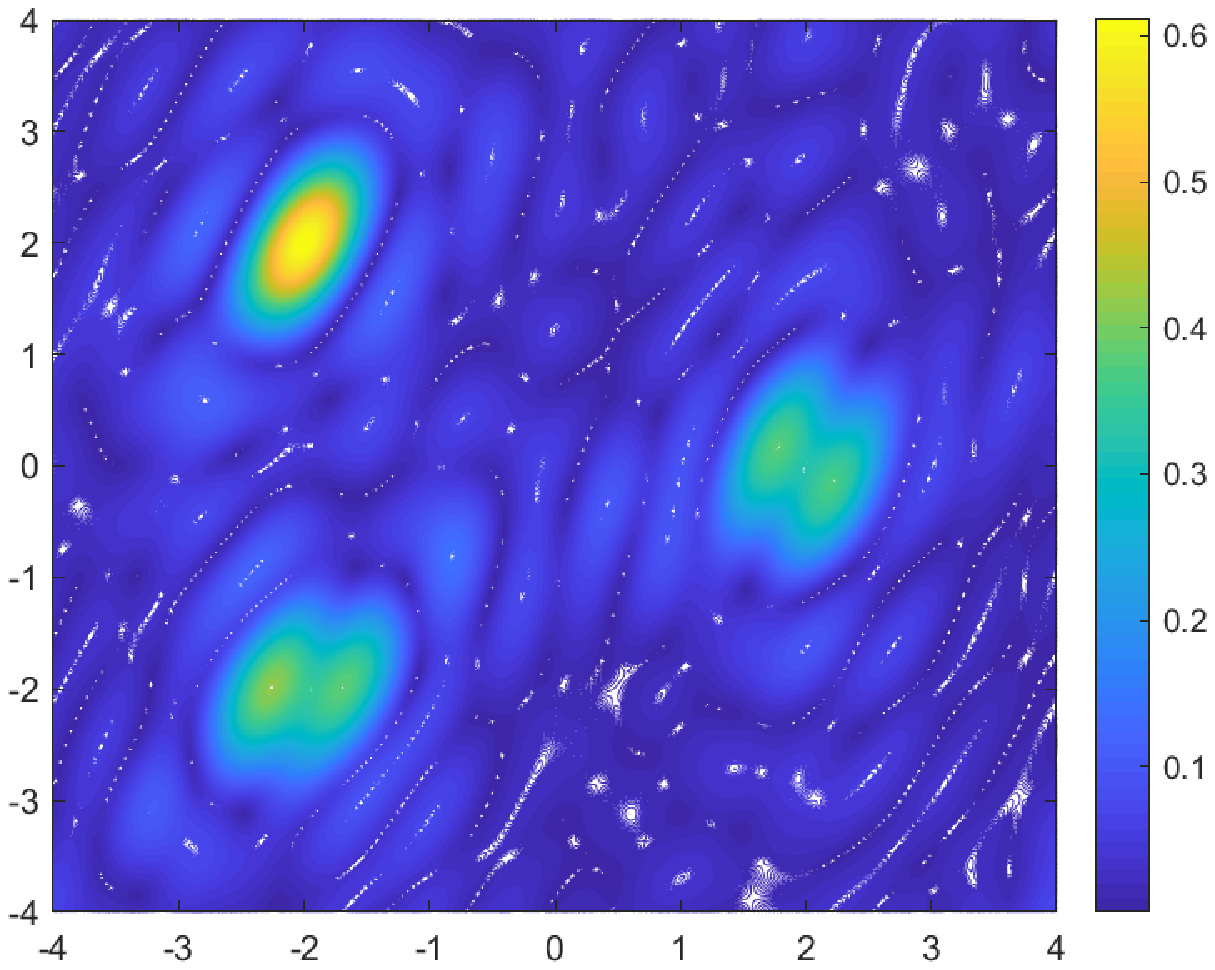}}
\end{tabular}
\end{center}
\caption{Indicator plots of DSM. Left: $S_1$. Middle: $S_2$. Right: $S_3$.}
\label{Fig2}
\end{figure}
The locations of sources $\{ z_{j}^{DSM} \}_{j=1}^3$ reconstructed  by the DSM are
\begin{eqnarray*}
&&S_1:\{(2.00,0.00),(-2.00,2.00),(-2.00,-2.00)\},\\
&&S_2: \{(2.00,-0.02),(-2.00,2.04),(-2.00,-2.02)\},\\
&&S_3: \{(1.98,0.00),(-2.00,2.00),(-2.20,-1.94)\}.
\end{eqnarray*}
For partial data, the reconstructed locations are not as accurate, but still satisfactory.

In the Bayesian inversion stage, we set $\beta=0.05$ and $\sigma=0.0004$. 
Then 20000 samples are drawn from the prior distribution and the first 5000 samples are discarded. 
The inversion results of the MCMC algorithm are shown in Table~\ref{e2t1}. 
In particular, for partial data, the reconstructed locations are improved in general.
\begin{table}[ht]
\small
    \centering
    \begin{tabular}{ l  c  c  c  c  c  c}
 \Xhline{1pt}
  \multirow{2}{*}{$j$} & \multicolumn{3}{c}{Exact Parameters} & \multicolumn{3}{c}{Reconstructed parameters for $S_1$}  \\  
  \cmidrule(lr){2-4}  \cmidrule(lr){5-7}
 & $\lambda_{j}$ & $\xi_{j}$ & $z_{j}$ & $\lambda_{j}^{\star}$ & $\xi_{j}^{\star}$ & $z_{j}^{\star}$ \\
 \Xhline{0.8pt}
  1 &  - &  $(\sqrt{2}, -\sqrt{2})$ & (2, 0) &  - & (1.4886, -1.4727) & (1.9936, -0.0019) \\
  2&   9 & - & (-2, 2) & 8.9373 & - & (-1.9863, 2.0037)\\
  3& - &  (2, 0) & (-2, -2) & - & (1.9696, -0.0350) & (-2.0030, -1.9970)\\
 \Xhline{1pt}
  \multirow{2}{*}{j } & \multicolumn{3}{c}{Reconstructed parameters for $S_2$} & \multicolumn{3}{c}{Reconstructed parameters for $S_3$}  \\  
  \cmidrule(lr){2-4}  \cmidrule(lr){5-7}
 & $\lambda_{j}^{\star}$ & $\xi_{j}^{\star}$ & $z_{j}^{\star}$ & $\lambda_{j}^{\star}$ & $\xi_{j}^{\star}$ & $z_{j}^{\star}$ \\
 \Xhline{0.8pt}
  1&  - &  (1.4087, -1.4573 ) & (1.9948, -0.0111) &  - & (1.4073, -1.1819) & (2.0037, 0.0072) \\
  2&     8.9656   & - & (-2.0077, 2.0116) &  7.7530 & - & (-1.9899, 2.0016)\\
  3 & - &  (1.7725, 0.0003) & (-2.0022, -2.0173) & - & (1.7091, 0.0544) & (-2.1487, -1.9156)\\
 \Xhline{1pt}
\end{tabular}
\caption{Exact and reconstructed parameters for Example 2.} 
\label{e2t1}
\end{table}

\subsection{Example 3} Let $F(x)=\sum_{j=1}^{2}\lambda_{j}\exp(-\xi_{j}\|x-z_{j}\|^2)$ with
\begin{eqnarray*}
&&\{\lambda_{j}\}_{j=1}^2=\{3,-4\},\\
&&\{\xi_{j}\}_{j=1}^2=\{2.5,1\},\\
&&\{z_{j}\}_{j=1}^2=\{(2,2),(-1.5,-1.5)\}.
\end{eqnarray*}
Set $[k_m, k_M]=[2,7]$ and $N_k=10$.
In Fig.~\ref{Fig3}, we plot the indicator functions for $S_1, S_2, S_3$. 
\begin{figure}[h!]
\begin{center}
\begin{tabular}{lll}
\resizebox{0.33\textwidth}{!}{\includegraphics{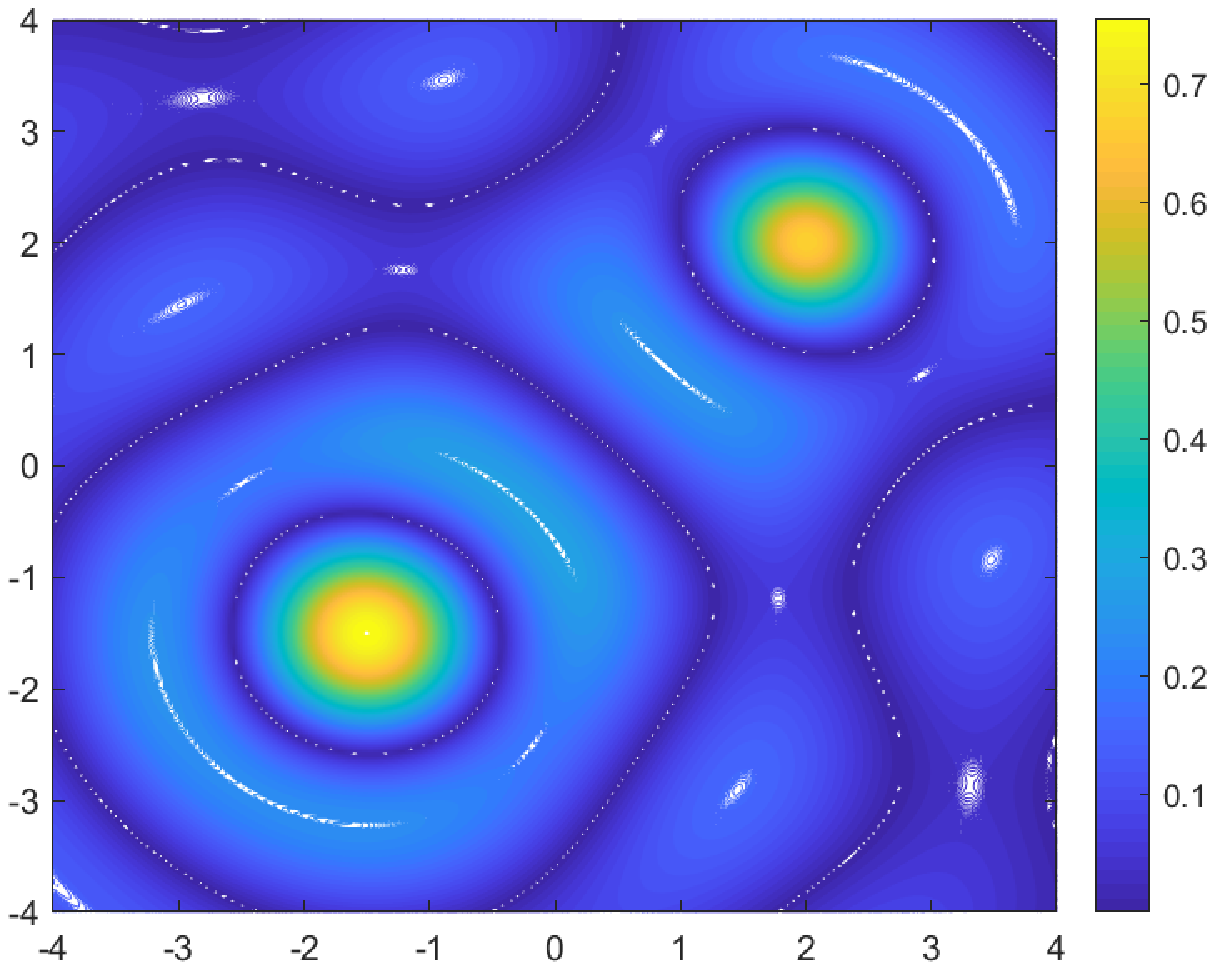}}&
\resizebox{0.33\textwidth}{!}{\includegraphics{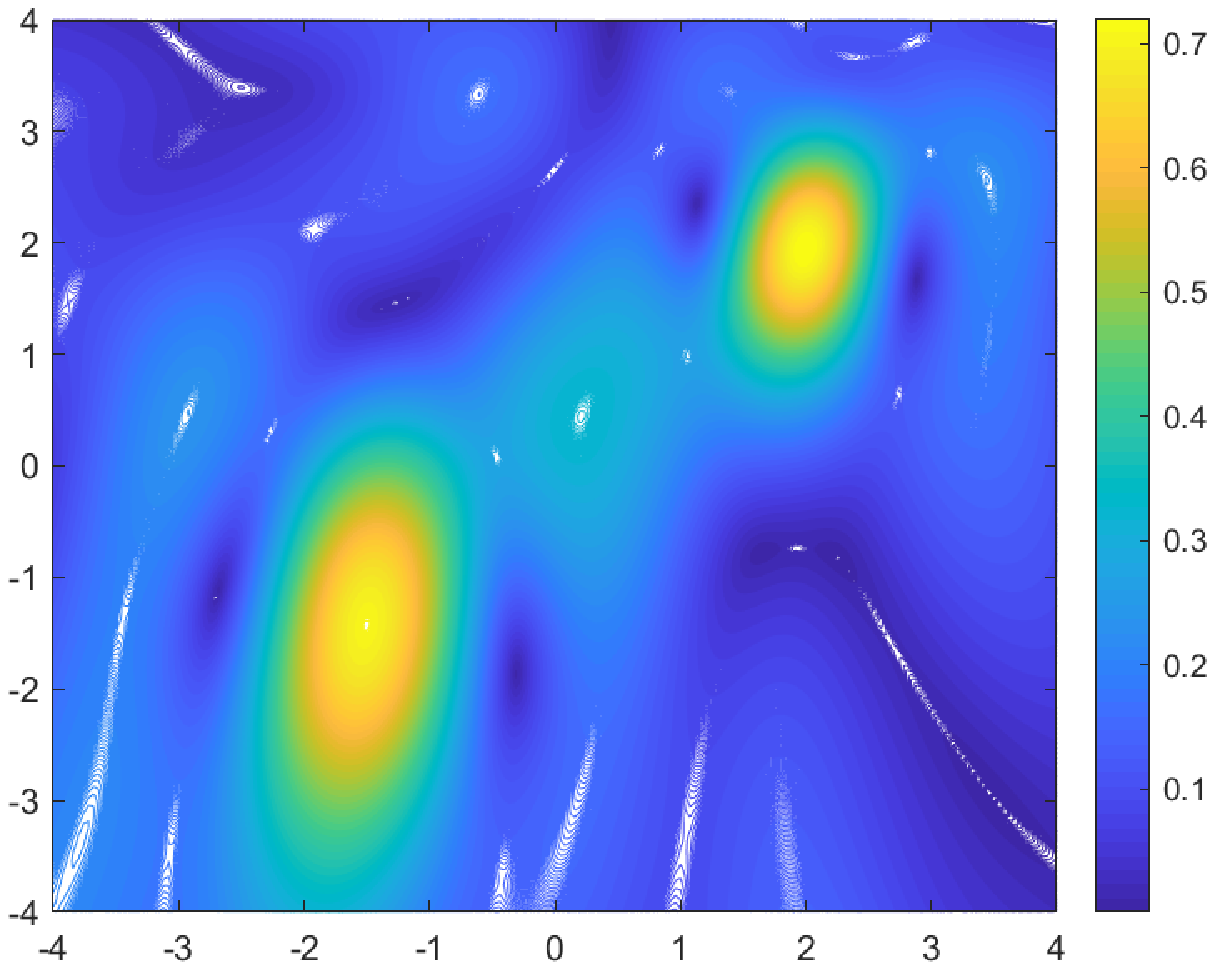}}&
\resizebox{0.33\textwidth}{!}{\includegraphics{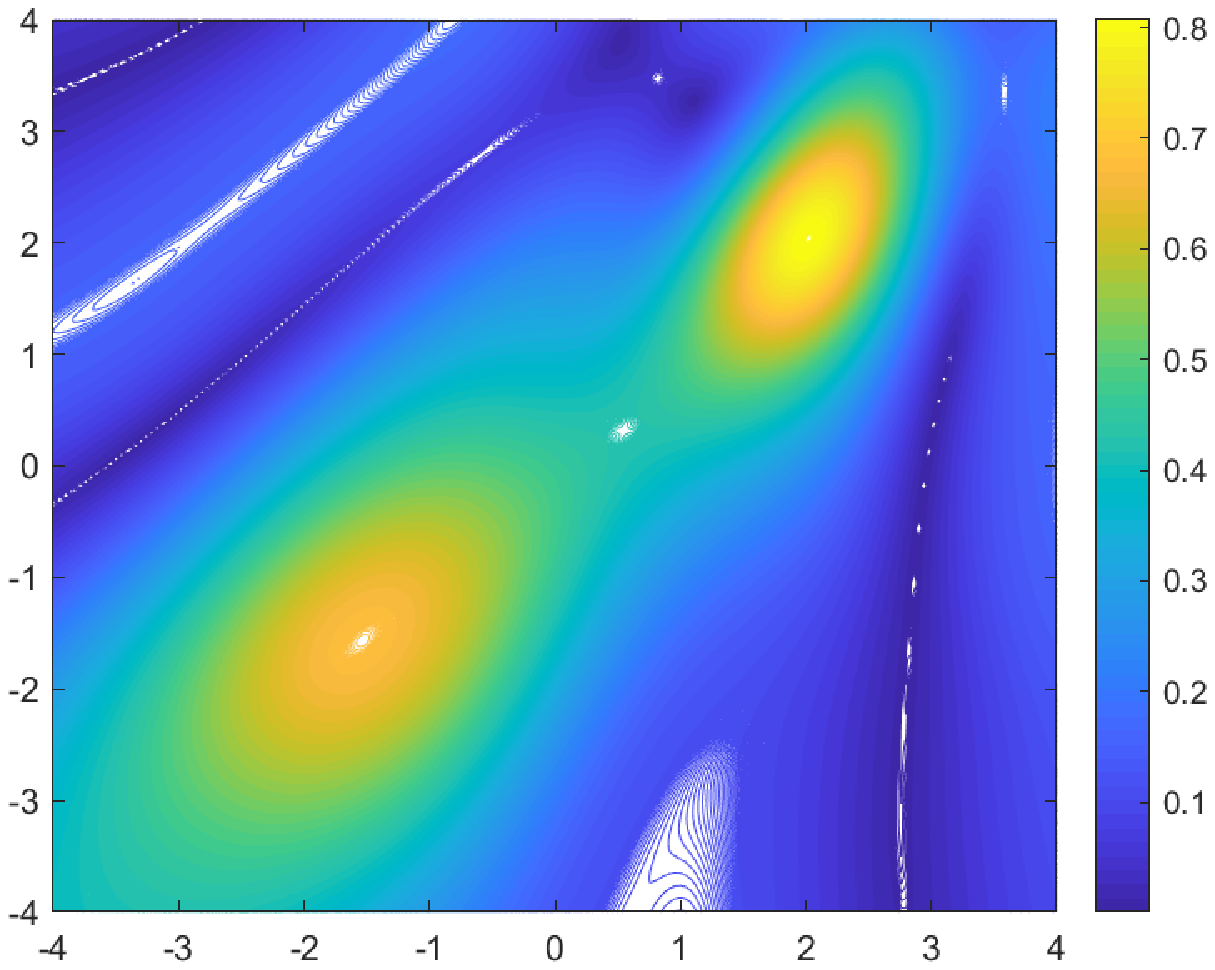}}
\end{tabular}
\end{center}
\caption{Indicator plots of DSM. Left: $S_1$. Middle: $S_2$. Right: $S_3$.}
\label{Fig3}
\end{figure}
The reconstructed locations $\{z_{j}^{DSM}\}_{j=1}^2$ by the DSM are
\begin{eqnarray*}
&&S_1:\{(2.00, 2.00),(-1.52, -1.52)\},\\
&&S_2:\{(2.00, 1.96),(-1.52, -1.44)\},\\
&&S_3:\{(2.04, 2.04),(-1.52, -1.56)\}.
\end{eqnarray*}
The locations are satisfactory. However, the results deteriorate for partial data.

In the Bayesian inversion stage, set $\beta=0.03$ and $\sigma=0.004$. 
15000 samples are drawn and the first 3000 samples are discarded. 
The inversion results are shown in Table~\ref{e3t1}. Overall, the reconstructions are satisfactory. 
Again, the Bayesian inversion significantly improves the reconstructed locations. 
\begin{table}[ht]
\small
    \centering
    \begin{tabular}{ l  c  c  c  c  c  c}
 \Xhline{1pt}
  \multirow{2}{*}{$j$} & \multicolumn{3}{c}{Exact Parameters} & \multicolumn{3}{c}{Reconstructed parameters for $S_1$}  \\  
  \cmidrule(lr){2-4}  \cmidrule(lr){5-7}
 & $\lambda_{j}$ & $\xi_{j}$ & $z_{j}$ & $\lambda_{j}^{\star}$ & $\xi_{j}^{\star}$ & $z_{j}^{\star}$ \\
 \Xhline{0.8pt}
 1 &  3 &  2.5 & (2, 2) &    3.0458 &  2.5154 & (1.9894, 2.0094) \\
  2 &   -4 & 1 & (-1.5, -1.5) & -4.0901 & 1.0202 & (-1.5012, -1.5128)\\
 \Xhline{1pt}
  \multirow{2}{*}{$j$} & \multicolumn{3}{c}{Reconstructed parameters for $S_2$}  & \multicolumn{3}{c}{Reconstructed parameters for $S_3$}  \\  
  \cmidrule(lr){2-4}  \cmidrule(lr){5-7}
 & $\lambda_{j}^{\star}$ & $\xi_{j}^{\star}$ & $z_{j}^{\star}$& $\lambda_{j}^{\star}$ & $\xi_{j}^{\star}$ & $z_{j}^{\star}$ \\
 \Xhline{0.8pt}
  1 & 2.8525 &  2.4827 & (2.0098, 1.9999) &   3.0494 & 2.3827 & (2.0211, 1.9951) \\
  2 & -4.1944 & 1.0136  & (-1.5002, -1.5067) & -3.9265 & 1.0216 & (-1.5064, -1.4779)\\
 \Xhline{1pt}
\end{tabular} 
\caption{Exact and reconstructed parameters for Example 3.} 
\label{e3t1}
\end{table}

\subsection{Example 4} Let $F(x)=\sum_{j=1}^{4}\lambda_{j}\exp(-\xi_{j}\|x-z_{j}\|^2)$ with
\begin{eqnarray*}
&& \{\lambda_{j}\}_{j=1}^4=\{2, 4, -3, 2.5\}, \\
&& \{\xi_{j}\}_{j=1}^4=\{2, 3, 2, 1\},\\
&& \{z_{j}\}_{j=1}^4=\{(2, 2), (-2, 2), (-2, -2), (2, -2)\}.
\end{eqnarray*}
Let $[k_m, k_M]=[1.5,8]$ and $N_k=15$.
In Fig.~\ref{Fig4}, we plot the indicator functions for $S_1, S_2, S_3$. 
\begin{figure}[h!]
\begin{center}
\begin{tabular}{lll}
\resizebox{0.33\textwidth}{!}{\includegraphics{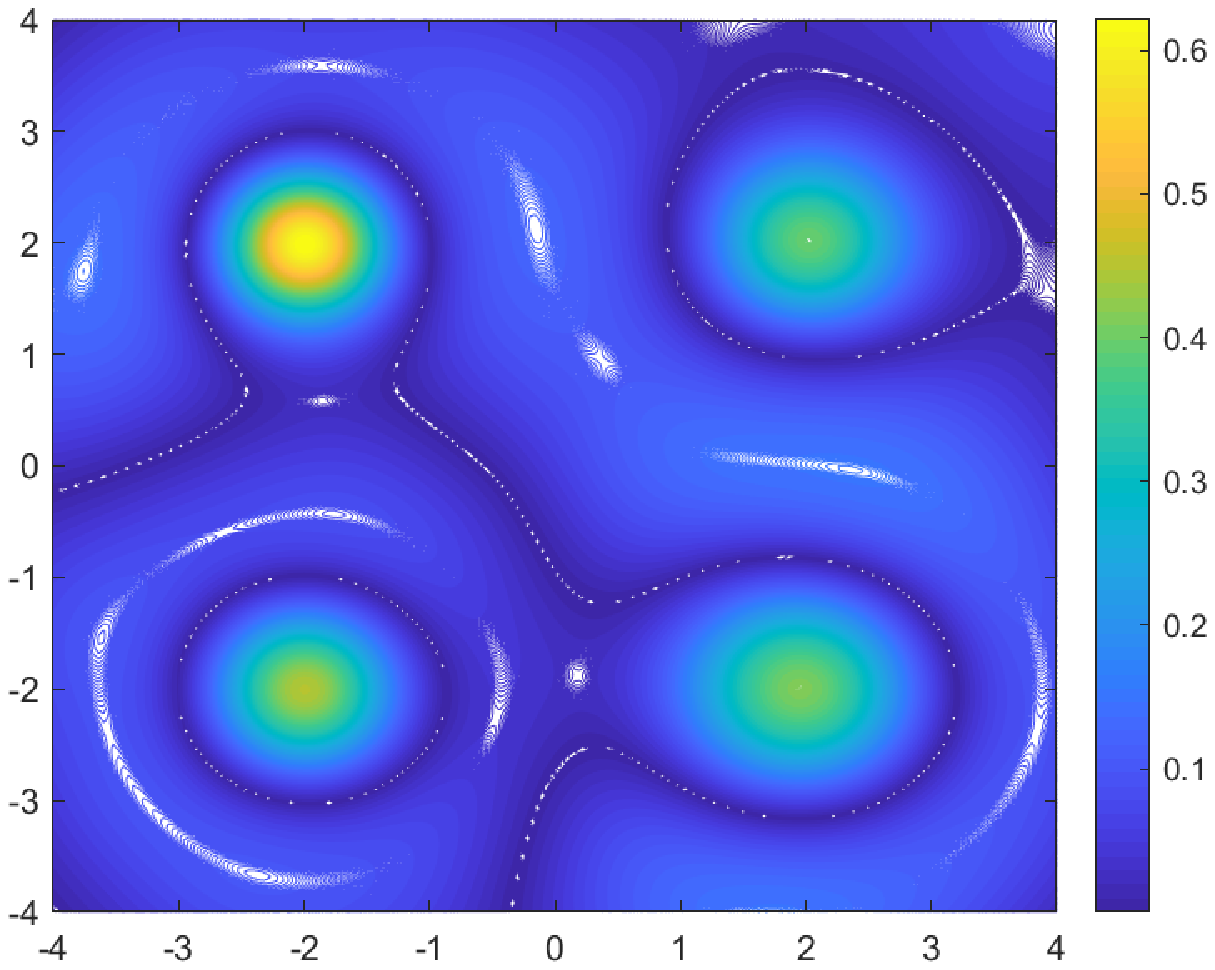}}&
\resizebox{0.33\textwidth}{!}{\includegraphics{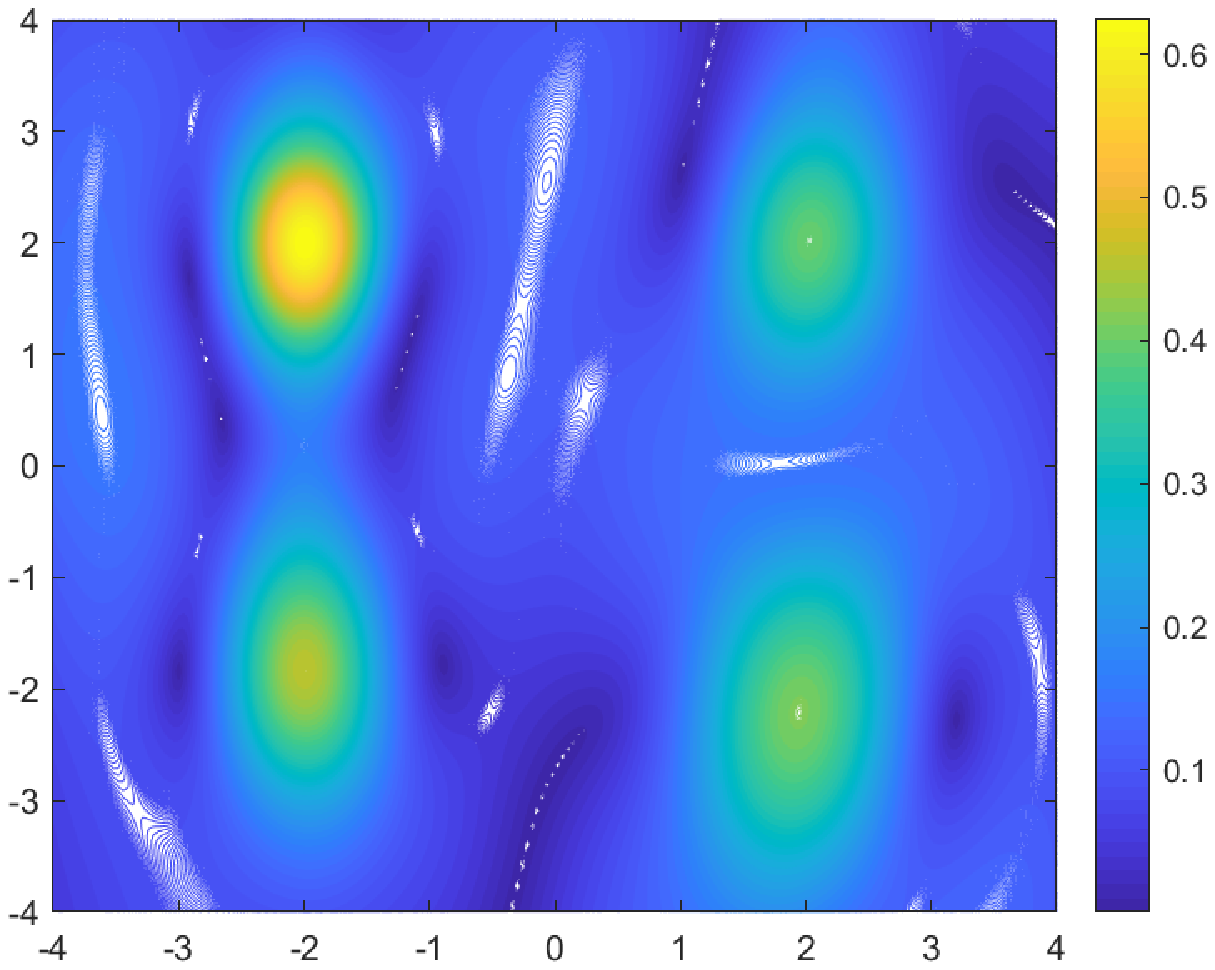}}&
\resizebox{0.33\textwidth}{!}{\includegraphics{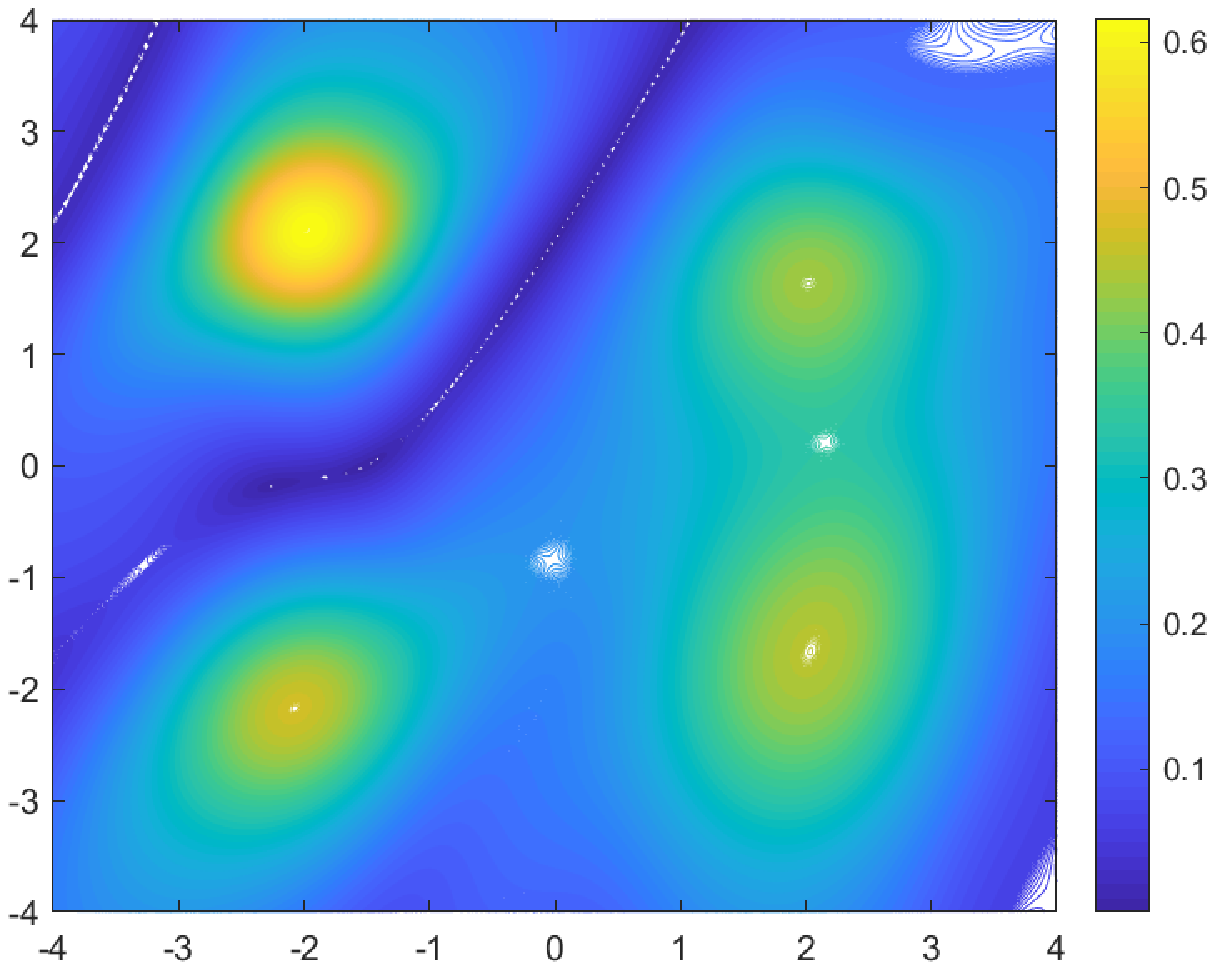}}
\end{tabular}
\end{center}
\caption{Indicator plots of DSM. Left: $S_1$. Middle: $S_2$. Right: $S_3$.}
\label{Fig4}
\end{figure} 

The reconstructed locations of sources $\{ z_{j}^{DSM} \}_{j=1}^4$ by the DSM are
\begin{eqnarray*}
&& S_1: \{(2.04,2.04),(-2.00,2.00),(-2.00,-2.00),(1.96,-2.00)\},\\
&& S_2: \{(2.04,2.04),(-2.00,2.00),(-2.00,-1.84),(1.96,-2.20)\},\\
&& S_3: \{(2.00,1.64),(-1.96,2.12),(-2.08,-2.16),(2.04,-1.64)\}.
\end{eqnarray*}
Similarly,  the results deteriorate when data becomes less.

In the Bayesian inversion stage, $\beta=0.03$ and $\sigma=0.004$. 
20000 samples are drawn and the first 5000 samples were discarded. 
The inversion results are shown in Table~\ref{e4t1}.
\begin{table}[ht]
\small
    \centering
    \begin{tabular}{ c  c  c  c  c  c  c}
 \Xhline{1pt}
  \multirow{2}{*}{$j$ } & \multicolumn{3}{c}{Exact Parameters} & \multicolumn{3}{c}{Reconstructed parameters for $S_1$}  \\  
  \cmidrule(lr){2-4}  \cmidrule(lr){5-7}
 & $\lambda_{j}$ & $\xi_{j}$ & $z_{j}$ & $\lambda_{j}^{\star}$ & $\xi_{j}^{\star}$ & $z_{j}^{\star}$ \\
 \Xhline{0.8pt}
  1 &  2 & 2 & (2, 2) &  2.2408 & 2.2712 & (2.0392, 1.9954) \\
  2 &   4 & 3 & (-2, 2) &  3.9250 & 2.9601 & (-2.0023, 2.0165)\\
  3 & -3 & 2 & (-2, -2) & -3.1428 & 2.1598 & (-1.9885, -2.0034)\\
  4 & 2.5 & 1 & (2, -2) &   2.5928 & 1.0771 & (1.9531, -2.0413)\\
 \Xhline{1pt}
  \multirow{2}{*}{$j$} & \multicolumn{3}{c}{Reconstructed parameters for $S_2$} & \multicolumn{3}{c}{Reconstructed parameters for $S_3$}  \\  
  \cmidrule(lr){2-4}  \cmidrule(lr){5-7}
 & $\lambda_{j}^{\star}$ & $\xi_{j}^{\star}$ & $z_{j}^{\star}$  & $\lambda_{j}^{\star}$ & $\xi_{j}^{\star}$ & $z_{j}^{\star}$ \\
 \Xhline{0.8pt}
  1 &  2.1579 & 2.1221 & (2.0347, 1.9412) &  1.9819   &  1.5675  & (2.0824, 1.7697  ) \\
  2 &   3.8821 & 2.8849 & (-1.9937, 2.0028) & 4.1515  & 2.9813  & (-2.0566, 2.0335 )\\
  3 & -2.8060 & 2.0517 & (-2.0167, -1.9650) & -3.2017 &   1.7718   & (-2.0677, -2.1082)\\
  4 & 2.3755 &  0.9439 & (1.9904, -2.1102) &   2.2924 & 0.8748 & (1.9733, -1.7340)\\
 \Xhline{1pt}
\end{tabular}
\caption{Exact and reconstructed parameters for Example 4.}
\label{e4t1}
\end{table}

\section{Conclusions}\label{conclusions}
A new quality-Bayesian approach is proposed to reconstruct the locations and intensities of the unknown acoustic sources. First, a direct sampling method is developed to approximate the locations of the sources. Second, the Bayesian inversion is used to get more detailed information. The locations of the sources obtained in the first step are coded in the priors.  A Metropolis-Hastings (MH) MCMC algorithm is employed to explore the posterior density. 

The DSM is fast for the qualitative information of the sources, while the Bayesian method is effective for the quantitative information. The approximate locations of the sources by the DSM are critical to the convergence of the MCMC method.  The two steps are based on the same physical model and use the same measured data. The new approach inherits the merits of both steps. In particular, when the locations by the DSM for partial data is not accurate, the results are significantly improved by the Bayesian inversion. Numerical examples show that the proposed method is effective, in particular, when the measurement data is partial.

\end{document}